\documentclass[a4paper,final]{article}

\usepackage{booktabs}
\usepackage{xypic}
\xyoption{all}
\SelectTips{cm}{10}

\usepackage{amsthm}
\newtheorem{teo}{Theorem}
\newtheorem{prop}[teo]{Proposition}

\newtheorem{cor}[teo]{Corollary}
\theoremstyle{definition}
\newtheorem{defi}[teo]{Definition}
\newtheorem{rem}[teo]{Remark}
\newtheorem{ex}[teo]{Example}
\newtheorem{conj}[teo]{Conjecture}

\usepackage{amssymb}

\usepackage{graphicx}
\usepackage{psfrag}
\psfrag{1}{\small $1$}
\psfrag{2}{\small $2$}
\psfrag{3}{\small $3$}
\psfrag{4}{\small $4$}
\psfrag{5}{\small $5$}
\psfrag{6}{\small $6$}
\psfrag{7}{\small $7$}
\psfrag{8}{\small $8$}
\psfrag{9}{\small $9$}

\usepackage{url}

\newcounter{listi}

\newcommand{\matRP}{\mathbb{RP}}

\newcommand{\calT}{{\cal T}}
\psfrag{T}{\small $\calT$}

\newcommand{\calR}{{\cal R}}
\newcommand{\calG}{{\cal G}}
\newcommand{\calD}{{\cal D}}
\newcommand{\calU}{{\cal U}}

\newcommand{\val}[1]{{\rm val}(#1)}
\newcommand{\valG}[1]{{\rm val}_{\calG}(#1)}
\newcommand{\degree}[1]{{\rm deg}(#1)}
\newcommand{\degG}[1]{{\rm deg}_{\calG}(#1)}
\newcommand{\mv}{{\rm mv}}
\newcommand{\md}{{\rm md}}

\newcommand{\st}[1]{{\rm st}(#1)}
\newcommand{\clst}[1]{{\rm clst}(#1)}
\newcommand{\link}[1]{{\rm link}(#1)}

\newcommand{\inter}{{\rm int}}

\newcommand{\T}{{\rm T}}

\newcommand{\phan}{\phantom{\Big|}}

\newcommand{\nI}{n_{\rm I}}
\newcommand{\nII}{n_{\rm II}}

\title{Decomposition and Enumeration\\of Triangulated Surfaces}

\author{Gennaro {\sc Amendola}\thanks{E-mail address: {\tt <amendola@mail.dm.unipi.it>}.}}

\begin{document}

\maketitle

\begin{abstract}
  We describe some theoretical results on triangulations of surfaces
  and we develop a theory on roots, decompositions and genus-surfaces.
  We apply this theory to describe an algorithm to list all triangulations of
  closed surfaces with at most a fixed number of vertices.
  We specialize the theory for the case where the number of vertices is
  at most~11 and we get theoretical restrictions on genus-surfaces allowing us to
  get the list of the triangulations of closed surfaces with at most~11
  vertices.
\end{abstract}

\vspace{.5cm}


\vspace{.5cm}


\section*{Introduction}

The enumeration of triangulations of surfaces ({\em i.e.}~simplicial complexes whose underlying topological space is a surface) was started by
Br\"uckner~\cite{Brueckner1897} at the end of the 19th century.
This study has been continued through the 20th century by many authors.
For instance, a complete classification of triangulations of closed surfaces with at
most 8 vertices was obtained by Datta~\cite{Datta1999},
and by Datta and Nilakantan~\cite{DattaNilakantan2002}, while the list of
such triangulations with at most 10 vertices was obtained by
Lutz~\cite{Lutz:10vert}.
The numbers of triangulations, depending on genus and number of vertices, are
collected in~\cite{Lutz:sito} and~\cite{Sulanke:sito}.

We point out that all these studies, as well as this paper, deal with genuine piecewise linear triangulations of surfaces, and not with mere gluings of triangles (for which different techniques should be used).

We will describe here an algorithm to list the triangulations of closed surfaces with at
most a fixed number of vertices.
This algorithm is based on some theoretical results which are interesting in
themselves.
By specializing this theory for the case where the number of vertices is at
most~11, we are able to improve the algorithm for this particular case.
We have hence written the computer program {\tt
  trialistgs11}~\cite{amendola:sito} giving a complete enumeration of all
triangulations of closed surfaces with at most 11 vertices.
Table~\ref{tab:tria_numbers} gives the detailed numbers of such triangulations.
This result has been obtained independently by Lutz and
Sulanke~\cite{Lutz-Sulanke:12vert}.
\begin{table}
\begin{center}
\begin{small}
\begin{tabular*}{\linewidth}{@{}l@{\extracolsep{12pt}}l@{\extracolsep{12pt}}r@{\extracolsep{12pt}}r@{\extracolsep{12pt}}r@{\extracolsep{\fill}}l@{\extracolsep{12pt}}l@{\extracolsep{12pt}}r@{\extracolsep{12pt}}r@{\extracolsep{12pt}}r@{}}
\toprule\\[-3mm]
\phantom{$V$} & \phantom{$S$} & \phantom{T} & \phantom{R} &
\phantom{N} &  \phantom{$V$} & \phantom{$S$}  & \phantom{T} &
\phantom{R} & \phantom{N} \\[-12pt]
$V$& $S$       &   T&   R&  N& $V$& $S$       &      T &      R &      N\\
\cmidrule{1-5}\cmidrule{6-10}
  4& $S^2$     &   1&   1&   &  10& $S^2$     &     233&      12&    221\\
   &           &    &    &   &    & $T^2$     &    2109&     887&   1222\\
  5& $S^2$     &   1&    &  1&    & $S^+_2$   &     865&     865&       \\
   &           &    &    &   &    & $S^+_3$   &      20&      20&       \\
  6& $S^2$     &   2&   1&  1&    & $\matRP^2$&    1210&     185&   1025\\
   & $\matRP^2$&   1&   1&   &    & $K^2$     &    4462&    1971&   2491\\
   &           &    &    &   &    & $S^-_3$   &   11784&    9385&   2399\\
  7& $S^2$     &   5&   1&  4&    & $S^-_4$   &   13657&   13067&    590\\
   & $T^2$     &   1&   1&   &    & $S^-_5$   &    7050&    7044&      6\\
   & $\matRP^2$&   3&   2&  1&    & $S^-_6$   &    1022&    1022&       \\
   &           &    &    &   &    & $S^-_7$   &      14&      14&       \\
  8& $S^2$     &  14&   2& 12&    &           &        &        &       \\
   & $T^2$     &   7&   6&  1&  11& $S^2$     &    1249&      34&   1215\\
   & $\matRP^2$&  16&   8&  8&    & $T^2$     &   37867&    9732&  28135\\
   & $K^2$     &   6&   6&   &    & $S^+_2$   &  113506&   93684&  19822\\
   &           &    &    &   &    & $S^+_3$   &   65878&   65546&    332\\
  9& $S^2$     &  50&   5& 45&    & $S^+_4$   &     821&     821&       \\
   & $T^2$     & 112&  75& 37&    & $\matRP^2$&   11719&    1050&  10669\\
   & $\matRP^2$& 134&  36& 98&    & $K^2$     &   86968&   23541&  63427\\
   & $K^2$     & 187& 133& 54&    & $S^-_3$   &  530278&  298323& 231955\\
   & $S^-_3$   & 133& 133&   &    & $S^-_4$   & 1628504& 1314000& 314504\\
   & $S^-_4$   &  37&  37&   &    & $S^-_5$   & 3355250& 3175312& 179938\\
   & $S^-_5$   &   2&   2&   &    & $S^-_6$   & 3623421& 3596214&  27207\\
   &           &    &    &   &    & $S^-_7$   & 1834160& 1833946&    214\\
   &           &    &    &   &    & $S^-_8$   &  295291&  295291&       \\
   &           &    &    &   &    & $S^-_9$   &    5982&    5982&       \\[1mm]
\bottomrule
\end{tabular*}
\end{small}
\end{center}
\caption{Number of triangulations~(T), roots~(R) and non-roots~(N), with at most 11
    vertices, depending on the number of vertices $V$ and on the closed surface $S$
    triangulated.}
  \label{tab:tria_numbers}
\end{table}

The aim of this paper is to describe the theory of what we call roots,
decompositions, and genus-surfaces, and to describe the algorithm based on this theory.
The implementation {\tt trialistgs11} of the algorithm is not
designed to be as fast as possible: more precisely, our program is slower than
Lutz-Sulanke's program~\cite{Lutz-Sulanke:12vert}.

A triangulation of a closed surface is a {\em root} if either it has no 3-valent vertex or it is
the boundary of the tetrahedron.
We will see that each triangulation of a closed surface can be transformed in a unique root by
repeatedly contracting edges containing a 3-valent vertex.
By uniqueness, roots divide the class of all triangulations of closed surfaces
into disjoint sub-classes, depending on their root.
One can think of roots as irreducible triangulations when only
edge-contractions deleting edges containing a 3-valent vertex are allowed.
Anyway, there are some differences; for instance, we gain uniqueness (in fact a
triangulation may have more than one irreducible triangulation), but we loose
finiteness (in fact we have infinitely many roots for each surface).

It is worth noting that the number of roots is by far smaller than the number of triangulations, at least as the number of vertices increases (see Table~\ref{tab:tria_numbers}).
Moreover, we note also that for the sphere $S^2$ the number of roots is very small, hence roots seem to work better for the sphere $S^2$ than for other surfaces.

Roughly speaking a {\em decomposition} of a closed triangulated surface is obtained by
dividing it into some disjoint triangulated discs and one triangulated surface (called
{\em genus-surface}) in such a way that at least one disc contains in its interior a
maximal-valence vertex of the triangulation.
Such a decomposition is called {\em minimal} if the number of triangles in the
genus-surface is the smallest possible one.
We will see that minimal decompositions fulfill many properties proved
theoretically.
Roughly speaking the algorithm consists of listing the pieces of such minimal
decompositions (by using the properties to simplify the search) and then
gluing the pieces found.

\subsection*{Definitions and notations}

From now on $S$ will always denote a connected compact surface.

\paragraph{Triangulated surfaces}
A {\em triangulation $\calT$ of a (connected compact) surface $S$} is a
simplicial complex whose underlying topological space is the surface $S$.
The vertices of the triangulation $\calT$ are usually denoted by numbers, say
$1,2,\ldots,n$; the choice of a (different) number for each vertex is called
{\em labeling}.
Obviously, the change of the labeling ({\em re-labeling}) modifies neither the
triangulation nor the surface.

When dealing with triangulations, there is the problem of deciding whether the
underlying topological space of a triangulation belongs to a particular class
(in our case, the class of surfaces).
This is in general a difficult matter, for instance there is {\em no}
algorithm to decide whether the underlying topological space of a given
$d$-dimensional simplicial complex is a $d$-sphere if $d\geqslant 5$.
In our case, in order to decide whether the underlying topological space
of a triangulation is a surface, we can check the property ``the link of each
vertex is a circle or an interval''.
The case of the interval is forbidden in the closed case.

Since we deal only with triangulations of surfaces, in order to define a
triangulation, it is enough to list the triangles.
Hence, for instance, the boundary of the tetrahedron can be encoded by ``$123\
124\ 134\ 234$'', see Fig.~\ref{fig:root_example}-left.
\begin{figure}
  \begin{center}
    \begin{tabular}{@{}c@{\extracolsep{0.6cm}}c@{\extracolsep{0.6cm}}c@{}}
      \begin{minipage}[t]{3.6cm}{\small{\begin{center}
              boundary
              
              of the tetrahedron
            \end{center}}}\end{minipage}
      &
      \begin{minipage}[t]{3.6cm}{\small{\begin{center}
              boundary
              
              of the octahedron
            \end{center}}}\end{minipage}
      &
      \begin{minipage}[t]{3.6cm}{\small{\begin{center}
              6-vertex $\matRP^2$
            \end{center}}}\end{minipage}
      \\
      \begin{minipage}[c]{3.6cm}{\small{\begin{center}
              \includegraphics{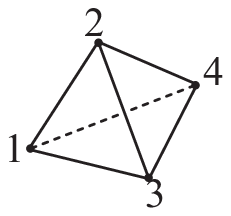} 
            \end{center}}}\end{minipage}
      &
      \begin{minipage}[c]{3.6cm}{\small{\begin{center}
              \includegraphics{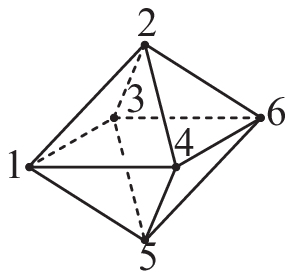}
            \end{center}}}\end{minipage}
      &
      \begin{minipage}[c]{3.6cm}{\small{\begin{center}
              \includegraphics{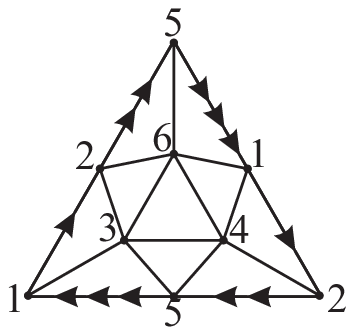} 
            \end{center}}}\end{minipage}
      \\
      \begin{minipage}[t]{3.6cm}{\small{\begin{center}
              123 124 134 234
            \end{center}}}\end{minipage}
      &
      \begin{minipage}[t]{3.6cm}{\small{\begin{center}
              123 124 135 145 236
              
              246 356 456
            \end{center}}}\end{minipage}
      &
      \begin{minipage}[t]{3.6cm}{\small{\begin{center}
              123 124 135 146 156
              
              236 245 256 345 346
            \end{center}}}\end{minipage}
    \end{tabular}
  \end{center}
  \caption{Examples of triangulations/roots.}
  \label{fig:root_example}
\end{figure}
Moreover, the order of the triangles in such lists can be changed arbitrarily,
hence it is not restrictive to choose always the lexicographically smallest
one.

It is well known~\cite{Rado1925} that each closed surface can be
{\em triangulated}, {\em i.e.}~it is the underlying topological space of a
simplicial complex.
This and the following paragraph allow us to forget about the abstract
surface and to use the term {\em triangulated surface}.

\paragraph{Euler characteristic}
For an arbitrary closed (orientable or non-orientable)
surface $S$ the Euler characteristic $\chi(S)$ of $S$ is the
alternating sum of the number of vertices $V(\calT)$, the number of edges
$E(\calT)$, and the number of triangles $T(\calT)$, {\em
  i.e.}~$\chi(S)=V(\calT)-E(\calT)+T(\calT)$, of any triangulation $\calT$ of $S$.
This definition makes sense because it turns out that it does not depend on the
triangulation $\calT$ but it depends only on the topological type of the surface $S$.

Since each triangle contains three edges and each edge is contained in two
triangles, we have $2E(\calT)=3T(\calT)$.
Thus, the number of vertices $V(\calT)$ and the Euler characteristic
$\chi(\calT)$ determine $E(\calT)$ and $T(\calT)$, by the formulae
$E(\calT)=3V(\calT)-3\chi(\calT)$ and $T(\calT)=2V(\calT)-2\chi(\calT)$.

A closed orientable surface $S^+_g$ of genus $g$ has Euler characteristic
$\chi(S^+_g)=2-2g$, whereas a closed non-orientable surface $S^-_g$ of genus $g$ has
Euler characteristic $\chi(S^-_g)=2-g$.
For instance, $S^+_0$ is the sphere $S^2$, $S^+_1$ is the torus $T^2$, $S^-_1$ is
the projective plane $\matRP^2$, $S^-_2$ is the Klein bottle $K^2$.
The topological type of a closed surface is completely determined if it is
known its Euler characteristic (or, equivalently, its genus) and whether it is
orientable or not; hence the notation $S^\pm_*$ above makes sense.

The smallest possible number of vertices $V(\calT)$ for a triangulation
$\calT$ of a closed surface $S$ is determined by Heawood's
bound~\cite{Heawood1890}
$$
V(\calT)\geqslant\Big\lceil{\small{\frac{1}{2}}\big(7+\sqrt{49-24\chi (S)}\big)}\Big{\rceil}.
$$
Ringel~\cite{Ringel1955} (for the non-orientable case), and then Jungerman and
Ringel~\cite{Jungerman-Ringel1980} (for the orientable case)
have proved that this bound is tight, except for $S^+_2$, the Klein bottle
$K^2$, and $S^-_3$, for each of which an extra vertex has to be added.

\paragraph{Notations}
Let now $\calT$ be a triangulation of a (non-necessarily closed) surface $S$.
We will denote by $\partial\calT$ the triangulation of the boundary of $S$
induced by $\calT$, and by $\inter(\calT)$ the triangulation of the
interior of $S$ induced by $\calT$.
If $S$ is closed, we have $\partial\calT=\emptyset$ and $\inter(\calT)=\calT$.
For each simplex $\sigma\in\calT$ we will denote
by $\st{\sigma}$ the open star of $\sigma$ ({\em i.e.}~the sub-triangulation of $\calT$ made up of the simplexes containing $\sigma$),
by $\clst{v}$ the closed star of $v$ ({\em i.e.}~the closure of $\st{\sigma}$),
and by $\link{v}$ the link of $v$ ({\em i.e.}~the sub-triangulation of $\clst{\sigma}$ made up of the simplexes disjoint from $\sigma$).
For each vertex $v\in\calT$ we will moreover denote by $\val{v}$ the valence of $v$
({\em i.e.}~the number of triangles of $\calT$ containing $v$) and by
$\degree{v}$ the degree of $v$ ({\em i.e.}~the number of vertices of $\calT$
adjacent to $v$.
When a sub-triangulation $\calU$ of $\calT$ will be considered, we will denote
by ${\rm val}_{\calU}(v)$ the valence of $v$ in $\calU$ and by ${\rm
  deg}_{\calU}(v)$ the degree of $v$ in $\calU$.
Note that $\degree{v}=\val{v}$ if $v\in\inter(\calT)$, while $\degree{v}=\val{v}+1$ if
$v\in\partial\calT$.
We will denote by $\mv(\calT)$ the maximal valence of the vertices of $\calT$
and by $\md(\calT)$ the maximal degree of the vertices of $\calT$.
Note that if $\calT$ is closed, $\md(\calT)=\mv(\calT)$.
With a slight abuse of notation we will freely intermingle between closed and
open triangles.

\begin{rem}\label{rem:small_tria}
  The boundary of the tetrahedron is the unique closed triangulated surface with maximal
  vertex-valence 3.
  The boundary of the octahedron, shown in Fig.~\ref{fig:root_example}-centre,
  is the unique closed triangulated surface with maximal vertex-valence 4 and
  without 3-valent vertices.
\end{rem}

\section{Roots}\label{sec:roots}

We will describe in this section the notion of root of a closed triangulated
surface.
Triangulated surfaces can be modified by applying a move called \T-{\em move}:
it consists in replacing an open triangle of the triangulation with the open star of
a new 3-valent vertex ({\em i.e.}~one new vertex, three new edges and three new
triangles), as shown in Fig.~\ref{fig:Tmove}.
Note that a \T-move can be applied for each triangle of a triangulated surface.
On the contrary, an {\em inverse} \T-move can be applied only if the triangulated surface $\calT$ has a 3-valent vertex and the link of this vertex does not bound already a triangle in $\calT$ (for otherwise the new triangle added by the inverse \T-move would already be in $\calT$); moreover, if the two conditions above are fulfilled, an inverse \T-move can be applied (the result being indeed a triangulated surface).
It is worth noting that the boundary of the tetrahedron is the only triangulated surface having a 3-valent vertex whose link bounds a triangle.

\begin{figure}
  \begin{center}
\psfrag{T}{\small \T}
    \includegraphics{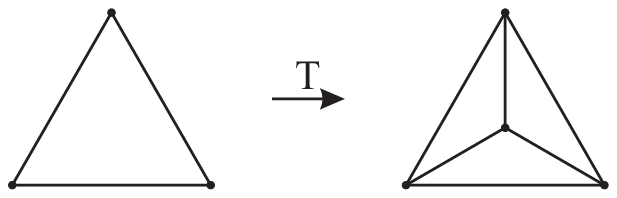}
  \end{center}  
  \caption{T-move.}
  \label{fig:Tmove}
\end{figure}

\begin{defi}
  A {\em root} of a closed triangulated surface $\calT$ is a triangulation
  $\calR$ obtained from $\calT$ by a sequence of inverse \T-moves and such
  that no inverse \T-move can be applied to it.
\end{defi}

\begin{ex}
  The boundary of the tetrahedron, the boundary of the octahedron, and the unique
  $\matRP^2$ with 6 vertices, shown in Fig.~\ref{fig:root_example}, are roots.
\end{ex}

\begin{rem}\label{rem:tetra_3}
  The boundary of the tetrahedron is the only root with a 3-valent vertex.
  In fact, as said above, when a closed triangulated surface has a 3-valent vertex, an inverse \T-move
  can be applied unless the added triangle is already in the triangulation,
  being then the boundary of the tetrahedron.
\end{rem}

\begin{rem}\label{rem:maxval_4}
  Remarks~\ref{rem:small_tria} and~\ref{rem:tetra_3} obviously imply that the boundary of the
  tetrahedron and the boundary of the octahedron are the only roots with
  maximal valence at most 4.
\end{rem}

Since \T-moves are particular edge-contractions, each irreducible closed
triangulated surface is a root.
But there are finitely many irreducible triangulations of each closed
surface~\cite{BarnetteEdelson1988}, while each closed surface $S$ has
infinitely many roots.
In fact, consider the boundary of the octahedron if $S=S^2$, or an irreducible
triangulation of $S$ otherwise; such triangulations are roots.
By repeatedly applying edge-expansions creating 4-valent vertices (as shown
in Fig.~\ref{fig:4val_create}), we get infinitely many different roots of
$S$ (they are roots because no 3-valent vertex appears).
\begin{figure}
  \centerline{\includegraphics{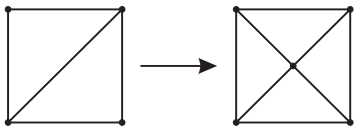}}
  \caption{How to create a new 4-valent vertex via an edge-expansion.}
  \label{fig:4val_create}
\end{figure}
It is worth noting that we have used the boundary of the octahedron (which
is not an irreducible triangulation) because we need a root without 3-valent
vertices, so that, when we apply edge-expansions creating 4-valent vertices,
we get no 3-valent vertex; while the sphere has only one irreducible
triangulation~\cite{Steinitz-Rademacher1934}, the boundary of the tetrahedron.

\begin{teo}\label{teo:root}
  Each closed triangulated surface has exactly one root.
\end{teo}

\begin{proof}
  Let $\calT$ be a closed triangulated surface.
  In order to prove the existence of a root for $\calT$, it is enough to
  repeatedly apply inverse \T-moves until it is possible, getting finally a
  root of $\calT$ (note that each inverse \T-move decreases by one
  the number of vertices, hence this procedure comes to an end).

  The proof of uniqueness is slightly longer.
  We prove it by induction on the length of the longest sequence of
  inverse \T-moves needed to get a root from $\calT$ (obviously,
  there is a longest one).
  If such a sequence has length 0, there is nothing to prove, in fact $\calT$
  is already a root and it has no other root, because inverse
  \T-moves cannot be applied to it.

  Suppose now that, if a closed triangulated surface $\calT$ has a root $\calR$ obtained
  from $\calT$ via a sequence having length $n$ and $n$ is the maximal length of
  such sequences, then $\calR$ is the unique root of $\calT$; and let us prove
  that, if a closed triangulated surface $\calT$ has a root $\calR$ obtained from $\calT$
  with a sequence having length $n+1$ and $n+1$ is the maximal length of such
  sequences, then $\calR$ is the unique root of $\calT$.
  In order to do this, consider the sequence
  $$
  \xymatrix{
    {\calT}
    \ar[r]^{m_1} &
    {\calT_1}
    \ar[r]^(.4){m_2} &
    {\phantom{\calT}\ldots\phantom{\calT}}
    \ar[r]^(.55){m_{n-1}} &
    {\calT_{n-1}}
    \ar[r]^(.55){m_{n}} &
    {\calT_{n}}
    \ar[r]^{m_{n+1}} &
    \calR
  }
  $$
  of inverse \T-moves relating $\calT$ to $\calR$, and suppose
  by contradiction that another sequence
  $$
  \xymatrix{
    {\calT}
    \ar[r]^{m'_1} &
    {\calT'_1}
    \ar[r]^(.4){m'_2} &
    {\phantom{\calT}\ldots\phantom{\calT}}
    \ar[r]^(.53){m'_{n'-1}} &
    {\calT_{n'-1}}
    \ar[r]^(.6){m_{n'}} &
    \calR'
  }
  $$
  of inverse \T-moves relating $\calT$ to another root $\calR'$
  exists (obviously, $n'\leqslant n+1$).
  Now, consider the two triangulations $\calT_1$ and $\calT'_1$, and note that
  the longest sequence of inverse \T-moves from each of them to the respective
  root has length at most $n$ (because otherwise we could find a sequence of
  inverse \T-moves from $\calT$ to a root with length greater than $n+1$).
  Hence, we can apply the inductive hypothesis and we have that $\calR$
  and $\calR'$ are the only roots of $\calT_1$ and $\calT'_1$,
  respectively.

  In order to prove that $\calR = \calR'$, the idea is to change the sequences
  used to obtain $\calR$ and $\calR'$.
  Let us call $v$ and $v'$ the (3-valent) vertices removed by the inverse
  \T-moves $m_1$ and $m'_1$, respectively.
  If $v=v'$, then $\calT = \calT'$ and hence $\calR = \calR'$.
  Therefore, we suppose $v\neq v'$.
  Note that $v$ and $v'$ are not adjacent, because otherwise $\calT$ would be
  the boundary of the tetrahedron (see Fig.~\ref{fig:2_adj_3val} and note that two
  vertices can be the endpoints of at most one edge) and this is not the case
  (for no inverse \T-move can be applied to the boundary of the tetrahedron).
  \begin{figure}
\psfrag{v}{\small $v$}
\psfrag{vp}{\small $v'$}
    \centerline{\includegraphics{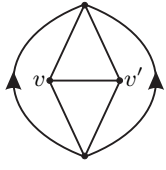}}
    \caption{If a closed triangulated surface contains two 3-valent adjacent vertices, it is
      the boundary of the tetrahedron.}
    \label{fig:2_adj_3val}
  \end{figure}
  Hence, we have $\st{v}\cap\st{v'}=\emptyset$ and the inverse \T-move
  removing $v'$ (resp.~$v$) can be applied to $\calT_1$ (resp.~$\calT'_1$);
  let us continue calling $m'_1$ (resp.~$m_1$) this move.
  In both cases we get the same triangulation, say $\calT''_2$.
  Now, let us consider a sequence
  $$
  \xymatrix{
    {\calT''_2}
    \ar[r]^{m''_3} &
    {\calT''_3}
    \ar[r]^(.4){m''_4} &
    {\phantom{\calT}\ldots\phantom{\calT}}
    \ar[r]^(.52){m''_{n''-1}} &
    {\calT_{n''-1}}
    \ar[r]^(.6){m''_{n''}} &
    \calR''
  }
  $$
  of inverse \T-moves relating $\calT''_2$ to a root $\calR''$, and
  let us use the sequences
  $$
  \xymatrix@R=0pt{
    {\calT_1} 
    \ar[rd]^{m'_1} & & & & & \\
    & {\calT''_2}
    \ar[r]^{m''_3} &
    {\calT''_3}
    \ar[r]^(.4){m'_4} &
    {\phantom{\calT}\ldots\phantom{\calT}}
    \ar[r]^(.52){m''_{n''-1}} &
    {\calT_{n''-1}}
    \ar[r]^(.6){m''_{n''}} &
    \calR'' \\
    {\calT'_1}
    \ar[ru]_{m_1} & & & & &
  }
  $$
  to obtain the roots $\calR$ and $\calR'$, which are equal to $\calR''$ by
  uniqueness.
  Hence, we have proved that $\calR = \calR'$, and we have done.
\end{proof}

\begin{rem}\label{rem:sub-classes}
  This theorem implies that the class of closed triangulated
  surfaces has a partition into (disjoint) sub-classes depending on their
  root.
  This fact implies that each invariant of a root, and in particular
  the root itself, is actually an
  invariant of all the closed triangulated surfaces having that root.
\end{rem}

\begin{rem}
  For the sake of completeness we will also prove that irreducible
  triangulations are in general not unique.
  More precisely, there are triangulations to which we can apply two
  edge-contractions leading to two different irreducible triangulations.
  Take for instance two different irreducible triangulations (say $\calT_1$
  and $\calT_2$) related by a {\em flip} ({\em i.e.}~a move modifying a square
  made up of two adjacent triangles by changing the diagonal, see
  Fig.~\ref{fig:two_irred}-below).
  The proof of the existence of such a pair can be found
  in~\cite{Sulanke2006:pre}.
  Now, consider the triangulation $\calT$ obtained by dividing the square of
  the flip by using both the diagonals, see Fig.~\ref{fig:two_irred}-above.
  We can apply two edge-contractions to $\calT$ leading to $\calT_1$ and
  $\calT_2$, respectively, as shown in Fig.~\ref{fig:two_irred}.
  \begin{figure}
\psfrag{T1}{\small $\calT_1$}
\psfrag{T2}{\small $\calT_2$}
\psfrag{flip}{\small flip}
    \centerline{\includegraphics{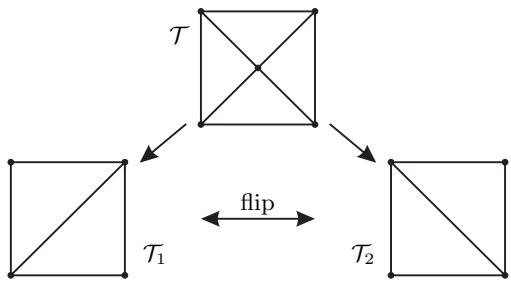}}
    \caption{Non-uniqueness of irreducible triangulations.
      In the lower part it is shown a flip, while the two moves above are
      edge-contractions leading to two different irreducible triangulations
      from $\calT$.}
    \label{fig:two_irred}
  \end{figure}
\end{rem}

We conclude this section by noting that roots could be used, for instance, to
try to prove the following
conjecture~\cite{Duke1970; Hougardy-Lutz-Zelke2006pre}.

\begin{conj}
  Each triangulation of the closed orientable surfaces with at most genus~$4$
  is realizable in ${\mathbb R}^3$ by straight edges, flat triangles, and
  without self-intersections.
\end{conj}

This conjecture is true for spheres (as proved by Steinitz~\cite{Steinitz1922},
and by Steinitz and Rademacher~\cite{Steinitz-Rademacher1934}), while its
natural extension to closed orientable surfaces of greater genus is not
true (as proved by Bokowski and Guedes de
Oliveira~\cite{Bokowski-Guedes_de_Oliveira2000}, and
by Schewe~\cite{Schewe2006pre}).
Obviously, it makes sense only for orientable closed surfaces, because
non-orientable closed surfaces are not embeddable in ${\mathbb R}^3$.
Note that, if a root is realizable, every triangulation having that root is
also realizable; hence, in order to prove the conjecture, it would be enough
to prove it only for roots.

\section{Genus-surfaces}\label{sec:genussurfs}

We will describe in this section the notion of decomposition and genus-surface
of a closed triangulated surface.

\begin{defi}
  A {\em decomposition} of a closed triangulated surface $\calT$ is a triple
  $(\calG,\calD,\{\calD_1,\ldots,\calD_n\})$, with $n\geqslant 0$, such that:
  \begin{itemize}
  \item
    $\calG,\calD,\calD_1,\ldots,\calD_n$ are sub-triangulations of $\calT$,
  \item
    $\calD,\calD_1,\ldots,\calD_n$ are triangulated discs,
  \item
    $\inter(\calD)$ contains a maximal-valence vertex of $\calT$,
  \item
    $\calG \cup \calD \cup \calD_1 \cup \ldots \cup \calD_n = \calT$,
  \item
    the intersection between each pair of these sub-triangulations is either a
    (triangulated) circle or empty.
  \end{itemize}
  The surface $\calG$ is called {\em genus-surface (of the decomposition)}, and
  $\calD$ is called {\em main disc (of the decomposition)}.
\end{defi}

First of all, we note that decompositions of closed triangulated surfaces exist.
In fact, if $\calT$ is a closed triangulated surface, then for each maximal-valence
vertex $v\in\calT$ we have that $(\calT\setminus\st{v},\clst{v},\emptyset)$ is
a decomposition of $\calT$.
We note also that the decompositions of a triangulation $\calT$ hold some
invariants of $\calT$.
For instance, the maximal valence of $\calT$ is the maximal valence of
internal vertices of the main disc only, and the genus of $\calT$ can
be computed from the genus-surface only.
As done for triangulations, two decompositions are considered
equivalent if they are obtained from each other by a re-labeling.

\begin{rem}
  Genus-surfaces are connected triangulated surfaces. In fact, they are
  obtained from closed triangulated surfaces by removing open triangulated discs whose closures are disjoint.
\end{rem}

\begin{defi}
  A decomposition $(\calG,\calD,\{\calD_1,\ldots,\calD_n\})$ of a closed
  triangulated surface $\calT$ is {\em minimal} if the number of triangles in $\calG$
  is minimal among all the decompositions of $\calT$.
  In such a case the genus-surface (of the decomposition) is also called {\em
    minimal}.
\end{defi}

By finiteness, minimal decompositions obviously exist.

\begin{ex}\label{ex:sphere_dec}
  The minimal decompositions of a triangulated sphere $\calT$ are exactly those of type
  $(\{T\},\calT\setminus\{T\},\emptyset)$, such that $T$ is a triangle of
  $\calT$ and a vertex not in $T$ has maximal valence (among those in
  $\calT$).
  Hence, in particular, the unique planar ({\em e.g.}~disc, annulus)
  minimal genus-surface is the triangle.  
\end{ex}

\begin{ex}\label{ex:mobius_strip}
  The unique decomposition of the (unique) $\matRP^2$ with
  6 vertices is
  $(\calT\setminus\st{v},\clst{v},\emptyset)$, where $v$ is any vertex of
  $\calT$.
  The genus-surface $\calT\setminus\st{v}$ is the M\"obius-strip with 5
  vertices, shown in Fig.~\ref{fig:moebius_strip}.
  \begin{figure}
    \centerline{\includegraphics{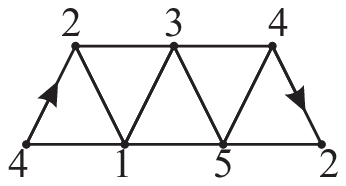}}
    \centerline{\small 123 124 135 245 345}
    \caption{The triangulation of the M\"obius-strip with 5 vertices.}
    \label{fig:moebius_strip}
  \end{figure}
\end{ex}

The following obvious remarks will be useful to make the search for
genus-surfaces of minimal decompositions faster.

\begin{rem}\label{rem:num_vert_gs}
  The inequality $V(\calG)\leqslant V(\calT)-1$ holds.
\end{rem}

\begin{rem}
  If $(\calG,\calD,\{\calD_1,\ldots,\calD_n\})$ and
  $(\calG',\calD',\{\calD'_1,\ldots,\calD'_{n'}\})$ are two decompositions of
  $\calT$ such that $\calG\subsetneq\calG'$, then
  $(\calG',\calD',\{\calD'_1,\ldots,\calD'_{n'}\})$ is not minimal.
\end{rem}

The following easy result will be also useful.

\begin{prop}\label{prop:vert-tria_D}
  The following inequalities hold:
  \begin{itemize}
  \item $T(\calD)\geqslant \mv(\calT)$,
  \item $V(\calD)\geqslant \mv(\calT)+1$.
  \end{itemize}
\end{prop}

\begin{proof}
  Note that the main disc $\calD$ contains a vertex $v\in\inter(\calD)$ with
  valence $\mv(\calT)$.
  Hence $\calD$ contains its closed star $\clst{v}$ and then at least
  $\mv(\calT)$ triangles and $\mv(\calT)+1$ vertices.
\end{proof}

\paragraph{Constructing the main disc}

It is easy to prove (so we leave it to the reader) that each main disc $\calD$
can be constructed from the closed star $\clst{v}$ of a maximal-valence vertex
$v\in\inter(\calD)$ by gluing repeatedly triangles along the boundary.
Such triangles can be glued in two ways: more precisely, along either one or
two edges of the boundary.
In the first case ({\em type}~I), the number of vertices in the boundary
increases by one, the number of vertices in the interior remains fixed, and
one 1-valent vertex in the boundary appears.
In the second case ({\em type}~II), the number of vertices in the boundary
decreases by one, the number of vertices in the interior increases by one, and
the valence of each vertex remaining in the boundary does not decrease.
See Fig.~\ref{fig:gluing_tria}.
\begin{figure}
\psfrag{D}{\small $\calD$}
\psfrag{I}{\small I}
\psfrag{II}{\small II}
  \centerline{\includegraphics{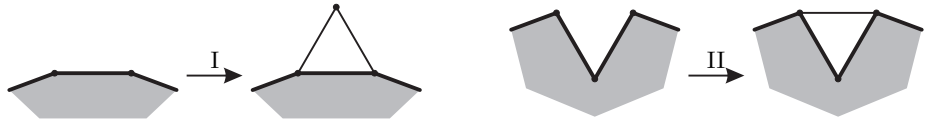}}
  \caption{Gluing triangles of type~I (on the left) and of type~II (on the right).}
  \label{fig:gluing_tria}
\end{figure}

\begin{rem}\label{rem:types_tria}
  Let us call $\nI$ (resp.~$\nII$) the number of triangles of type~I
  (resp. type~II) glued to obtain $\calD$.
  The following properties hold.
  \begin{enumerate}
  \item\label{pt:vbndD}
    We have $V(\partial\calD)=\mv(\calT)+\nI-\nII$.
  \item\label{pt:vintD}
    We have $V(\inter(\calD))=1+\nII$.
  \item\label{pt:nI_nII}
    If $\calT$ is a root, then the first triangle glued to $\clst{v}$ must be of
    type~I (because otherwise a 3-valent vertex in $\inter(\calD)$ would
    appear), and hence $\nII>0$ implies $\nI>0$.
    More precisely, we have ruled out the boundary of the tetrahedron (having
    3-valent vertices), but it has only one decomposition and we have
    $\nI=\nII=0$ for it, however.
  \end{enumerate}
  Note also that properties~\ref{pt:vbndD} and~\ref{pt:vintD} above
  imply that $\nI$ and $\nII$ are defined unambiguously, whatever the
  maximal-valence vertex $v$ in $\inter(\calD)$ and the order of the gluings are.
\end{rem}

This simple remark allows us to find other restrictions on genus-surfaces.
\begin{prop}\label{prop:lenght_bnd_condition}
  Suppose $\calT$ is a non-sphere root.
  Let us call~{\rm (A)} the condition
  ``$V(\calG)=V(\calT)-1$''
  and~{\rm (B)} the condition
  ``the maximal degree of a vertex in $\calG$ is achieved in
  $\partial\calG\cap\calD$.''
  Then, under the hypotheses made above, the number of vertices in the boundary
  of the main disc is bounded from below as described by the following
  table.
  
  \centerline{
  \begin{tabular}{|c|c|c|}
    \hline
    \phan (A) & (B) & $V(\partial\calD)\geqslant$ \\ \hline
    \phan true & true & $\md(\calG) + 1$ \\
    \phan true & false & $\md(\calG)$ \\
    \phan false & true & $\md(\calG)+3+V(\calG)-V(\calT)$. \\
    \phan false & false & $\md(\calG)+2+V(\calG)-V(\calT)$. \\ \hline
  \end{tabular}}
\end{prop}

\begin{proof}
  As done above, let us call $\nI$ (resp.~$\nII$) the number of
  triangles of type~I (resp.~of type~II) glued to obtain $\calD$.

  Suppose that condition~(A) does hold.
  In such a case we have $V(\inter(\calD))=1$, hence, by
  applying Remark~\ref{rem:types_tria}.\ref{pt:vintD} and then
  Remark~\ref{rem:types_tria}.\ref{pt:vbndD}, we get
  $V(\partial\calD)=\mv(\calT)+\nI\geqslant \mv(\calT)$.
  If condition~(B) does not hold, we have done because
  $\mv(\calT)\geqslant\md(\calG)$.
  Suppose hence that condition~(B) does hold.
  Let us call $w$ a vertex in $\partial\calG\cap\calD$ with maximal degree
  in $\calG$.
  Now, the valence of $w$ in $\calD$ is either~1 or greater than~1.
  In the first case, we have $\nI>0$ and hence
  $V(\partial\calD) \geqslant \mv(\calT)+1 \geqslant \md(\calG)+1$.
  In the second case, we have $\mv(\calT) \geqslant \val{w} > \md(\calG)$ and
  hence the thesis.

  Suppose now that condition~(A) does not hold.
  In this case, by Remark~\ref{rem:num_vert_gs}, we have
  $V(\calG)<V(\calT)-1$, or equivalently $0\geqslant
  2+V(\calG)-V(\calT)$.
  By Remark~\ref{rem:types_tria}, we get that the number of vertices in
  $\partial\calD$ is $\mv(\calT)+\nI-\nII$ and that $\nII =
  V(\inter(\calD))-1\leqslant V(\calT)-V(\calG)-1$.
  Now we have two cases to be analyzed: either $\nI=0$ or $\nI>0$.
  
  We analyze the case where $\nI=0$ first.
  By Remark~\ref{rem:types_tria}.\ref{pt:nI_nII}, we have $\nII=0$ and hence
  $\calD$ is the closed star of a maximal-valence vertex of $\calT$.
  If now condition~(B) does not hold, we have 
  $V(\partial\calD) = \mv(\calT) \geqslant \md(\calG) \geqslant
  \md(\calG)+2+V(\calG)-V(\calT)$.
  On the contrary, if condition~(B) does hold, let us call $w$ a vertex
  in $\partial\calG\cap\calD$ with maximal degree in $\calG$; the degree of
  $w$ in $\calT$ is $\md(\calG)+1$, hence
  $V(\partial\calD) = \mv(\calT) \geqslant \val{w} = \md(\calG)+1 \geqslant
  \md(\calG)+3+V(\calG)-V(\calT)$.

  We are left to prove the thesis in the case where $\nI\geqslant
  1$.
  If condition~(B) does not hold, we have
  $V(\partial\calD) = \mv(\calT)+\nI-\nII \geqslant
  \md(\calG)+2+V(\calG)-V(\calT)$.
  Finally, suppose that condition~(B) does hold.
  Let us call $w$ a vertex in $\partial\calG\cap\calD$ with maximal degree
  in $\calG$.
  If $\nII=0$, we have
  $V(\partial\calD) = \mv(\calT)+\nI \geqslant
  \mv(\calT)+1 \geqslant \md(\calG)+3+V(\calG)-V(\calT)$.
  On the contrary, suppose $\nII>0$.
  We have two cases to be analyzed, in both of which we get the thesis.
  In fact,
  \begin{itemize}
  \item[-] if the valence of $w$ in $\calD$ is one, there are at least two
    triangles of type~I (for otherwise it would exist a 3-valent vertex in
    $\calT$), and hence
    $V(\partial\calD) = \mv(\calT)+\nI-\nII \geqslant
    \md(\calG)+3+V(\calG)-V(\calT)$;
  \item[-] if the valence of $w$ in $\calD$ is greater than one, we have
    $\mv(\calT)>\md(\calG)$, and hence
    $V(\partial\calD) = \mv(\calT)+\nI-\nII >
    \md(\calG)+2+V(\calG)-V(\calT)$.
  \end{itemize}
\end{proof}

\subsection{Restrictions on minimal genus-surfaces}
\label{sec:restrictions}

Minimal genus-surfaces satisfy many restrictions: we now describe some of
them.
Throughout this section, $(\calG,\calD,\{\calD_1,\ldots,\calD_n\})$ will
always denote a decomposition of a closed triangulated surface $\calT$.

\begin{prop}\label{prop:link_edge_bnd}
  If $(\calG,\calD,\{\calD_1,\ldots,\calD_n\})$ is minimal, the link of each
  edge of $\partial\calG$ is made up of a vertex also contained in
  $\partial\calG$.
\end{prop}

\begin{proof}
  First of all, we note that the link of an edge contained in the boundary of a
  surface contains exactly one vertex.
  Now, suppose by contradiction that $e\subset\partial\calG$ is an edge
  such that $\link{e}=\{v\}$, where $v\in\inter(\calG)$.
  If we remove the triangle containing $e$ and $v$ from $\calG$ and add it to
  the disc (either $\calD$ or $\calD_i$, for some $i\in\{1,\ldots,n\}$) to which it is adjacent (in $\calT$),
  we get a decomposition of $\calT$ whose genus-surface has one triangle less
  than $\calG$: {\em i.e.}~a contradiction to the hypothesis that
  $(\calG,\calD,\{\calD_1,\ldots,\calD_n\})$ is minimal.
\end{proof}

\begin{prop}
  If $(\calG,\calD,\{\calD_1,\ldots,\calD_n\})$ is minimal, each triangle in
  $\calG$ intersects the boundary of $\calG$.
\end{prop}

\begin{proof}
  Suppose by contradiction that a triangle (say $T$) of $\calG$ does not
  intersect the boundary of $\calG$.
  Then $(\calG\setminus\{T\},\calD,\{\calD_1,\ldots,\calD_n,\{T\}\})$
  is a decomposition of $\calT$ whose genus-surface $\calG\setminus\{T\}$ has
  one triangle less than $\calG$: {\em i.e.}~a contradiction to the hypothesis
  that $(\calG,\calD,\{\calD_1,\ldots,\calD_n\})$ is minimal.
\end{proof}

\begin{prop}
  If $(\calG,\calD,\{\calD_1,\ldots,\calD_n\})$ is minimal, then $\calG$
  contains no pair of triangles adjacent to each other along an edge and
  adjacent along edges to different boundary components of $\calG$ (see
  Fig.~{\rm \ref{fig:no_pair}}-left).
  \begin{figure}
\psfrag{C1bS}{\small $C_1\subset\partial\calG$}
\psfrag{C2bS}{\small $C_2\subset\partial\calG$}
\psfrag{G}{\small $\calG$}
\psfrag{Gp}{\small $\calG'$}
    \centerline{\includegraphics{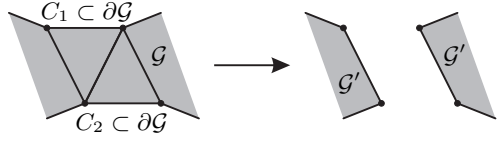}}
    \caption{Forbidden configuration if $C_1$ and $C_2$ are different boundary
      components of $\calG$ (on the left), because of the existence of a
      genus-surface $\calG'$ with fewer triangles than $\calG$ (on the right).}
    \label{fig:no_pair}
  \end{figure}
\end{prop}

\begin{proof}
  Suppose by contradiction that such a pair exists.
  If we remove these two triangles from $\calG$ and glue (by adding the two
  triangles) the two discs (different by hypothesis) adjacent to $\calG$ along the two boundary
  components adjacent to the two triangles, we get a decomposition of $\calT$
  whose genus-surface $\calG'$ has two triangles less than $\calG$ (see
  Fig.~\ref{fig:no_pair}-right): {\em i.e.}~a contradiction to the hypothesis
  that $(\calG,\calD,\{\calD_1,\ldots,\calD_n\})$ is minimal.
\end{proof}

\begin{rem}
In the proof of the proposition above we need the boundary components (adjacent to the pair of triangles) to be different; for otherwise the operation of {\em removing and gluing} would lead to an annulus in the complement of the genus-surface.
Note that the two boundary components may be the same, see for instance Example~\ref{ex:mobius_strip}.
\end{rem}

\begin{prop}\label{prop:deg_val_vert_S}
  Suppose $\calT$ is a non-sphere root and
  $(\calG,\calD,\{\calD_1,\ldots,\calD_n\})$ is minimal.
  Then, for each vertex $v$ of $\calG$, the following inequalities hold:
  \begin{itemize}
  \item $3 \leqslant \degG{v} \leqslant V(\calT)-2$ if $v\in\partial\calG$,
  \item $4 \leqslant \degG{v} \leqslant V(\calT)-2$ if $v\in\inter(\calG)$;
  \end{itemize}
  or equivalently:
  \begin{itemize}
  \item $2 \leqslant \valG{v} \leqslant V(\calT)-3$ if $v\in\partial\calG$,
  \item $4 \leqslant \valG{v} \leqslant V(\calT)-2$ if $v\in\inter(\calG)$.
  \end{itemize}
  Moreover, there exists at least one vertex $w$ in $\partial\calG$ with
  $\degG{w} \geqslant 4$ (or equivalently $\valG{w} \geqslant 3$).
\end{prop}

\begin{proof}
  Since
  $\valG{v}=
  \left\{\begin{small}
      \begin{array}{ll}
        \degG{v}-1 & {\rm if}\ v\in\partial\calG \\
        \degG{v} & {\rm if}\ v\in\inter(\calG)
      \end{array}
    \end{small}\right.$,
  it is enough to prove the following inequalities:
  \begin{list}{({\rm \alph{listi}})}{\usecounter{listi} \setlength{\labelwidth}{1cm}}
  \item $\degG{v} \leqslant V(\calT)-2$ for each $v\in\calG$,
  \item $4 \leqslant \degG{v}$ if $v\in\inter(\calG)$,
  \item $2 \leqslant \valG{v}$ if $v\in\partial\calG$,
  \item $\valG{w} \geqslant 3$ for at least one vertex $w\in\partial\calG$.
  \end{list}

  Inequality~(a) is obvious because $\degG{v}+1\leqslant V(\calG)\leqslant V(\calT)-1$ by
  Remark~\ref{rem:num_vert_gs}.
  Inequality~(b) is also obvious by Remark~\ref{rem:tetra_3}, because $\calT$
  is a non-sphere root.

  We will now prove inequality~(c).
  Suppose by contradiction that $\valG{v}=1$ for a vertex
  $v\in\partial\calG$.
  Hence, there is exactly one triangle $T\subset\calG$ such that $v\in
  T$.
  The two edges of $T$ incident to $v$ belong to $\partial\calG$;
  therefore, the third edge of $T$ does not belong to $\partial\calG$ because
  $\calG$ is not a disc, and we can remove the triangle $T$ from $\calG$
  adding it to the disc (either $\calD$ or $\calD_i$, for some
  $i\in\{1,\ldots,n\}$) to which it is adjacent (in $\calT$).
  We have got a decomposition of $\calT$ whose genus-surface has one triangle
  less than $\calG$: {\em i.e.}~a contradiction to the hypothesis that
  $(\calG,\calD,\{\calD_1,\ldots,\calD_n\})$ is minimal.

  Let us finally prove inequality~(d).
  Suppose by contradiction that $\valG{v}<3$ for each vertex
  $v\in\partial\calG$.
  By inequality~(c), we have $\valG{v}=2$ for each vertex
  $v\in\partial\calG$.
  Let us consider one of these vertices (say $w$).
  It is contained in two triangles (say $T_1$ and $T_2$), and it is adjacent
  to three vertices (call $v_0$ the one contained in $T_1\cap T_2$ and $v_i$
  the one contained in $T_i$ only, for $i=1,2$).
  By Proposition~\ref{prop:link_edge_bnd}, we have $v_0\in\partial\calG$
  and hence $\valG{v_0}=2$.
  Therefore, $v_0$ is contained in $T_1$ and $T_2$ only.
  This implies that $\calG = T_1 \cup T_2$ and hence that $\calG$ is a disc,
  contradicting the hypothesis that $\calT$ is not a sphere.
\end{proof}

\section{Listing closed triangulated surfaces}

In this section, we will apply the theory of roots and decompositions
to find an algorithm to list
all triangulations of closed surfaces with at most $n$ vertices.
Then we will specialize it (by specializing the theory described above) for the
case where $n=11$.
In fact, we will see that minimal decompositions of closed triangulated surfaces with at most 11
vertices satisfy some stronger theoretical restrictions.
Closed triangulated surfaces with at most $12$ vertices has been independently listed
by Lutz and Sulanke~\cite{Lutz-Sulanke:12vert} using a very subtle lexicographic
enumeration approach.
The computer program they have written to implement their algorithm is very fast.
Perhaps an algorithm mixing their technique and ours may be even faster.

\subsection{The listing algorithm}\label{sec:listing_algorithm}

First of all, we recall that we have seen in Section~\ref{sec:roots} that each
closed triangulated surface (say $\calT$) has exactly one root (Theorem~\ref{teo:root}).
Hence, in order to list closed triangulated surfaces, it is enough to list first roots and
then non-roots (deducing the latter ones from the former ones).
Moreover, we have seen in Section~\ref{sec:genussurfs} that each closed triangulated surface
has a (minimal) decomposition, say $(\calG,\calD,\{\calD_1,\ldots,\calD_n\})$.
Hence, we can start listing triangulated discs and genus-surfaces, and then we can
glue them to get all closed triangulated surfaces.
When the triangulated surface is a root and the decomposition is minimal, the
latter one
satisfies some theoretical restrictions which simplify the search.
However, there is a drawback slowing down the search: such a decomposition is
not unique.

\paragraph{Classical listing technique}
The basic technique we use to list triangulations is the classical one.
We start from the closed star of a maximal-valence vertex $v$ and we
repeatedly glue triangles.
Each time, we choose an edge in the boundary and we glue a triangle along that
edge.
In order to do this, we only need to choose the third vertex of the triangle:
this vertex can be either a vertex of the current boundary or a new one.
For each choice we create a new triangulation and we repeat the procedure for
it.
If, at some time, we violate some property which must be satisfied
({\em e.g.}~the first vertex is maximal-valent, the link of each vertex is
always contained in a circle), we go back trying to glue other triangles.
This technique has been described in details by Lutz in~\cite{Lutz:10vert}, and
by Lutz and Sulanke in~\cite{Lutz-Sulanke:12vert}.

\paragraph{Re-labeling}
In order to avoid duplicates, each time we find a triangulated surface $\calT$,
we check whether $\calT$ has been already found, changing the labeling to get
the mixed-lexicographically smallest one.
A list of triangles is {\em mixed-lexicographically smaller} than another one
if the first vertex has greater valence or, when the first vertices have the
same valence, the list is smaller in the lexicographic order.

In a mixed-lexicographically minimal triangulation, the list of triangles
starts with the star of vertex~1:
$$123\ \ 124\ \ 135\ \ 146\ \ 157\ \ \dots\ \ 1(d-1)(d+1),$$
where $d=\degree{1}$.
Now, there are two possibilities:
\begin{itemize}
\item if vertex~1 is in the interior of $\calT$, then its link must be a
  circle and hence the next triangle is $1d(d+1)$;
\item if otherwise vertex~1 is in the boundary of $\calT$, then its link is an
  interval and hence no more triangles containing vertex~1 appear in the
  list.
\end{itemize}
Then the list continues with the remaining triangles, not containing
vertex~1.

Note that, when we list a class of triangulations following the
mixed-lexico\-graphic order, the sub-class of triangulations with the same
maximal valence (say $m$) are sorted lexicographically, and every
triangulation in such a sub-class begins with the same $m$ triangles.
But note that the list of all triangulations does not follow the lexicographic
order.

In order to find the mixed-lexicographically smallest labeling, we carry out the
following steps:
\begin{itemize}
\item we list all maximal-valence vertices of $\calT$;
\item for each such vertex (say $v$) we list the vertices in $\link{v}$;
\item for each pair $(v,w)$, where $w\in\link{v}$, we re-label $v$ as $1$ and
  $w$ as $2$;
\item we re-label the two vertices in the link of the edge $vw$ to
  be $3$ and $4$ (we have two choices);
\item we extend the new labeling in a lexicographically smallest way (sometimes we
  have two choices as in the previous step);
\item we search, among all such pairs $(v,w)$, for the
  mixed-lexicograph\-ically smallest labeling.
\end{itemize}

\paragraph{The listing algorithm}
The algorithm is made up of 5 steps.
Let $n$ be the maximal number of vertices of the closed triangulated surfaces
we are searching for.
\begin{itemize}
\item[1.]
{\em Triangulated discs}

The list of triangulated discs can be achieved by applying the classical technique
described above.
Obviously, we want to list roots, hence we discard triangulated discs with 3-valent
vertices in their interior.
Listing of triangulated discs with at most $n$ vertices (with $n$ small)
is fast and well known, hence we do not describe this step.
We have essentially used the same technique of Step 3 below.

\item[2.]
{\em Triangulated spheres}

Minimal decompositions of triangulated spheres are of type
$(\{T\},\calT\setminus\{T\},\emptyset)$, where $T$ is a triangle of
$\calT$ and a vertex not in $T$ has maximal valence (among those in $\calT$),
see Example~\ref{ex:sphere_dec}.
Obviously, we have $V(\partial(\calT\setminus\{T\}))=3$, hence, in order to list
triangulated spheres, we pick out the triangulated discs $\calD$ such that
$V(\partial\calD)=3$ and we glue the missing triangle to each of them.
Moreover, each main disc has a maximal-valence vertex in its interior, so we
discard all triangulated spheres with maximal valence greater than this
number.
Finally, we note that we need to re-label each triangulated sphere found to
check that it has not been already found.

\item[3.]
{\em ``Minimal'' genus-surfaces}

In order to list ``minimal'' genus-surfaces, we follow the same classical technique, but
we have some restrictions that minimal genus-surfaces of roots must fulfill (see, for
instance, the results of Section~\ref{sec:roots} and~\ref{sec:genussurfs}),
hence we can discard those not fulfilling these restrictions.
Note that we will not know whether all genus-surfaces found are actually
minimal: we know only that they fulfill some restrictions necessary to be
minimal and that all minimal genus-surfaces are found.
Moreover, the number of the genus-surfaces we will found may be greater than that of
closed triangulated surfaces/roots, but this search has the advantage of dealing with a lower
number of vertices (at most $n-1$) and triangles, being necessary to construct
genus-surfaces instead of closed triangulated surfaces (see Remark~\ref{rem:num_vert_gs} and
Proposition~\ref{prop:vert-tria_D}).
Finally, we note that, as above, we need to re-label each genus-surface
found to check that it has not been already found.

\item[4.]
{\em Gluings}

In order to get roots from genus-surfaces, we glue the triangulated discs found to each
genus-surface found (along its boundary components).
One of such discs is a main disc, hence it contains a maximal-valence vertex in its interior.
Note that we must check all possible gluings between the genus-surface and
the triangulated disc(s): the number of the possible gluings of each disc is twice the
number of its boundary vertices.
Note that we need to check that the result of the gluing is a triangulated
surface, which essentially means that we must check that each pair of
adjacent vertices of the genus-surface is not adjacent in the triangulated discs.
Moreover, each main disc has a maximal-valence vertex in its interior, so we
discard the triangulations with maximal valence greater than this number.
Note also that, since we are searching for roots, we discard the
triangulations with a 3-valent vertex.
Finally, we note that, as above, we need to re-label each root found to check
that it has not been already found.

\item[5.]
{\em Non-roots}

We know that non-roots can be divided depending on their root (see
Remark~\ref{rem:sub-classes}).
Hence, we start from each root (with at most $n-1$ vertices) and we list
the non-roots having that root.
The search is quite simple: in order to get all non-roots from a root, it is
enough to repeatedly apply \T-moves.
Obviously, we need to re-label each non-root found to check that it has not
been already found, but the search can be restricted to those having the same
root only.
\end{itemize}

\subsection{Listing closed triangulated surfaces with at most 11 vertices}

If the maximal number of vertices of the triangulations we are searching for
is at most~11, we can make the listing algorithm faster, because in this case
minimal decompositions of roots fulfill restrictions stronger than in the
general case.
Throughout this section, $(\calG,\calD,\{\calD_1,\ldots,\calD_n\})$
will always denote a decomposition of a closed triangulated surface $\calT$.
Recall that, by Remark~\ref{rem:num_vert_gs}, if $V(\calT)\leqslant 11$ then
$V(\calG)\leqslant 10$.

We have already noted that we have no useful restriction {\em a~priori} on the
number of discs in $\calT\setminus\calG$.
If instead $V(\calT)\leqslant 11$, we have the following result.

\begin{prop}\label{prop:2comp}
  If $V(\calT)\leqslant 11$, the number of boundary components of $\calG$
  is~$1$ or~$2$.
\end{prop}

\begin{proof}
  By definition, $\calG$ has non-empty boundary, hence it is enough to prove
  that $\partial\calG$ has at most 2 components.
  The main ingredient of the proof is that each component must contain at
  least 3 vertices because it is a triangulated circle.

  Let we firstly suppose that $\mv(\calT)\leqslant 4$.
  By Remark~\ref{rem:maxval_4}, the root of $\calT$ is the boundary of the
  tetrahedron or the boundary of the octahedron.
  In order to get $\calT$ from them, we must apply \T-moves.
  It is very easy to prove that only 3 possibilities for $\calT$ arise: the
  two roots and one non-root with 5 vertices.
  By Remark~\ref{rem:num_vert_gs}, we have $V(\calG) \leqslant V(\calT)-1
  \leqslant 5$, and hence $\calG$ has at most one boundary component.

  Suppose now that $\mv(\calT)>4$.
  By Proposition~\ref{prop:vert-tria_D}, we have $V(\calD) \geqslant
  \mv(\calT)+1 \geqslant 6$ and then $V(\partial\calG\setminus\calD) \leqslant
  V(\calT)-V(\calD) \leqslant 5$.
  Hence there cannot be more than 2 components in $\partial\calG$, for
  otherwise $\partial\calG\setminus\calD$ would have at least 2 components
  containing at least 6 vertices.
\end{proof}

The following results will prove that, if $\calT$ is a root with at most $11$
vertices and $(\calG,\calD,\{\calD_1,\ldots,\calD_n\})$ is minimal, then only
three cases for $\calT\setminus(\calG\cup\calD)$ may arise.

\begin{prop}\label{prop:length_D1}
  If $V(\calT)\leqslant 11$, each minimal decomposition of $\calT$ is of
  type $(\calG,\calD,\emptyset)$ or of type $(\calG,\calD,\{\calD_1\})$ where
  $V(\partial\calD_1)$ is $3$ or $4$.
\end{prop}

\begin{proof}
By Example~\ref{ex:sphere_dec}, all minimal decompositions of a triangulated
sphere are of type $(\{T\},\calT\setminus\{T\},\emptyset)$, where $T$ is a
triangle; hence the proposition is obvious for triangulated spheres.
Therefore, suppose $\calT$ is not a triangulated sphere (recall that
$\mv(\calT)\geqslant 5$ by Remark~\ref{rem:maxval_4}).
By Proposition~\ref{prop:2comp}, we have that $\partial\calG$ has either one or two
boundary components.
In the first case, there is nothing to prove.
In the second case, the decomposition is $(\calG,\calD,\{\calD_1\})$.
We suppose by contradiction that $V(\partial\calD_1)>4$.
By Proposition~\ref{prop:vert-tria_D}, we have $V(\calD) \geqslant
\mv(\calT)+1 \geqslant 6$ and $V(\calD_1)\leqslant V(\calT)-V(\calD) \leqslant
5$; hence $V(\calD_1)=5$.
We also have $6 \leqslant \mv(\calT)+1 \leqslant V(\calD) \leqslant
V(\calT)-V(\partial\calD_1) \leqslant 6$, hence $V(\calD)=6$, $\mv(\calT)=5$,
and $\calD=\clst{v}$, where $v$ is the maximal-valence vertex of $\calD$.
Moreover, $V(\inter(\calD_1)) \leqslant V(\calT)-V(\calD)-V(\partial\calD_1)
\leqslant 0$, hence $V(\inter(\calD_1))=0$ and $\calD_1$ is the triangulation shown
in Fig.~\ref{fig:no_int_tria}.
\begin{figure}
  \centerline{\includegraphics{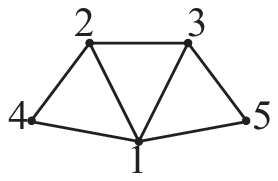}}
  \centerline{\small 123 124 135}
  \caption{The unique triangulated surface $\calD_1$ with $V(\partial\calD_1)=5$ and
    $V(\inter(\calD_1))=0$.}
  \label{fig:no_int_tria}
\end{figure}

Let us call $w$ the unique 3-valent vertex of $\calD_1$.
Since $\mv(\calT)=5$, we have that the valence of $w$ in $\calG$ is 2 (it
cannot be 1 by Proposition~\ref{prop:deg_val_vert_S}).
Let us call $T_1$ and $T_2$ the two triangles in $\calG$ incident to $w$, and
$w'$ the other vertex in $T_1\cap T_2$.
See Fig.~\ref{fig:bnd_D1}.
\begin{figure}
  \psfrag{v}{\small $v$}
  \psfrag{D}{\small $\calD$}
  \psfrag{D1}{\small $\calD_1$}
  \psfrag{w}{\small $w$}
  \psfrag{wp}{\small $w'$}
  \psfrag{T1}{\small $T_1$}
  \psfrag{T2}{\small $T_2$}
  \psfrag{G}{\small $\calG$}
  \centerline{\includegraphics{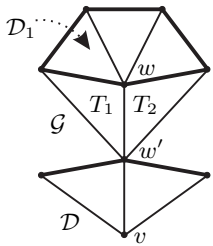}}
  \caption{If $V(\partial\calD_1)>4$, there is a 6-valent vertex in $\partial
    D$.}
  \label{fig:bnd_D1}
\end{figure}
By Proposition~\ref{prop:link_edge_bnd}, we have $w'\in\partial\calG$;
moreover, since $w$ is adjacent in $\calD_1$ to all the other boundary
vertices of $\calD_1$, then $w'$ does not belong to $\calD_1$ (for, otherwise,
$\calT$ would not be a triangulation of a surface); hence $w'\in\calD$.
The edge of $T_1$ not incident to $w$ is not contained in
$\partial\calD$ (for, otherwise, $\calD$ and $\calD_1$ would have non-empty
intersection); the same holds for $T_2$, hence the valence of $w'$ in $\calG$
is at least $4$.
Finally, the valence of $w'$ in $\calD$ is $2$, because $\calD=\clst{v}$.
Hence $\val{w'}\geqslant 6$, see Fig.~\ref{fig:bnd_D1}: this is a
contradiction to $\mv(\calT)=5$.
\end{proof}

\begin{prop}
  If $\calT$ is a root and $V(\calT)\leqslant 11$, then each minimal
  decomposition of $\calT$ is of type $(\calG,\calD,\emptyset)$ or of type
  $(\calG,\calD,\{\calD_1\})$ where $V(\inter(\calD_1))=0$.
\end{prop}

\begin{proof}
By Example~\ref{ex:sphere_dec}, all minimal decompositions of a triangulated
sphere are of type $(\{T\},\calT\setminus\{T\},\emptyset)$, where $T$ is a
triangle; hence the proposition is obvious for triangulated spheres.
Therefore, suppose $\calT$ is not a triangulated sphere (recall that
$\mv(\calT)\geqslant 5$ by Remark~\ref{rem:maxval_4}).
By Proposition~\ref{prop:length_D1}, we have three possibilities for the
decomposition of $\calT$.
It can be of type $(\calG,\calD,\emptyset)$, of type $(\calG,\calD,\{\calD_1\})$ where
$V(\partial\calD_1)=4$, or of type $(\calG,\calD,\{\calD_1\})$ where
$V(\partial\calD_1)=3$.
For the first type, there is nothing to prove.

Let us analyze the case where the decomposition is of type
$(\calG,\calD,\{\calD_1\})$ and $V(\partial\calD_1)=4$.
Suppose by contradiction that $V(\inter(\calD_1))>0$.
Since $V(\partial\calD) \geqslant 3$, we have $V(\inter(\calD)) \leqslant
V(\calT)-V(\partial\calD_1)-V(\inter(\calD_1))-V(\partial\calD) \leqslant 3$.
We will now use notations (and results) of Remark~\ref{rem:types_tria}.
We have $\nII\leqslant 2$, and then
$$
V(\partial\calD)=\mv(\calT)+\nI-\nII\geqslant\left\{
  \begin{array}{ll}
    \mv(\calT)\geqslant 5 & \mbox{if $\nI=0$} \\[2pt]
    \mv(\calT)+1-\nII\geqslant 4 & \mbox{if $\nI>0$}
  \end{array}
\right..
$$
Now, we can repeat this technique, getting $V(\inter(\calD)) \leqslant
V(\calT)-V(\partial\calD_1)-V(\inter(\calD_1))-V(\partial\calD) \leqslant 2$,
$\nII\leqslant 1$, and then
$$
V(\partial\calD)=\mv(\calT)+\nI-\nII\geqslant\left\{
  \begin{array}{ll}
    \mv(\calT)\geqslant 5 & \mbox{if $\nI=0$} \\[2pt]
    \mv(\calT)+1-\nII\geqslant 5 & \mbox{if $\nI>0$}
  \end{array}
\right..
$$
Hence, we get $V(\partial\calD)\geqslant 5$, $V(\inter(\calD))=1$, and then
$\calD=\clst{v}$, where $v$ is the maximal-valence vertex of $\calD$.
Moreover, we have $V(\calD) \leqslant V(\calT)-V(\partial\calD_1)-V(\inter(\calD_1))
\leqslant 6$ and hence $\val{v}=\mv(\calT)=5$.

We also have $\calD_1=\clst{w}$, where $w$ is the vertex in
$\inter(\calD_1)$, because $V(\inter(\calD_1)) \leqslant
V(\calT)-V(\calD)-V(\partial\calD_1) \leqslant 1$, $\val{w}\geqslant 4$, and
$V(\partial\calD_1)=4$.
Now, note that $V(\inter(\calG))=0$ and that each vertex in $\partial\calG$ has
valence $2$ or $3$ (for the valence of these vertices in $\calD$ or $\calD_1$
is $2$, the maximal valence in $\calT$ is $5$, and the valence in $\calG$ is
at least $2$ by Proposition~\ref{prop:deg_val_vert_S}).
With the same technique used in the end of the proof of
Proposition~\ref{prop:length_D1}, we can prove that no vertex in
$\partial\calG$ has valence $2$ (the only difference is in the proof that $w$
and $w'$ do not belong to the same disc $\calD$ or $\calD_1$, where the
property $2\leqslant\valG{v}\ \forall v\in\inter(\calG)$ from
Proposition~\ref{prop:deg_val_vert_S} should be used).
Hence, all vertices in $\partial\calG$ have valence $3$.
This and the fact that $V(\inter(\calG))=0$ easily imply that $\calG$ is the
annulus shown in Fig.~\ref{fig:annulus_9}: a contradiction to
Example~\ref{ex:sphere_dec}.
\begin{figure}
  \centerline{\includegraphics{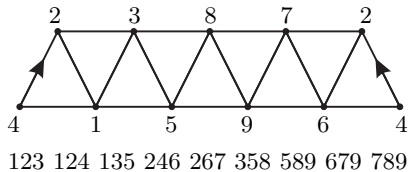}}
  \centerline{\small 123 124 135 246 267 358 589 679 789}
  \caption{The unique genus-surface $\calG$ such that $V(\inter(\calG))=0$,
    $\valG{w}=3$ for each vertex $w\in\partial\calG$, and with two boundary
    components with $4$ and $5$ vertices respectively.}
  \label{fig:annulus_9}
\end{figure}

Let us analyze now the case where the decomposition is of type
$(\calG,\calD,\{\calD_1\})$ and $V(\partial\calD_1)=3$.
Suppose by contradiction that $V(\inter(\calD_1))>0$.
Let us call $w$ a vertex in $\inter(\calD_1)$.
Obviously, we have $\val{w}\geqslant 4$.
By applying the same technique used to construct main discs (see
Remark~\ref{rem:types_tria}, also for notations), we easily get $\nI>0$
and $\nII = \val{w}+\nI-V(\partial\calD_1) \geqslant 2$.
Hence, $V(\inter(\calD_1)) = 1+\nII\geqslant 3$, $V(\calD_1)\geqslant 6$, and
$V(\inter(\calD)) \leqslant V(\calT)-V(\calD_1)-V(\partial\calD) \leqslant 2$.

We now apply Remark~\ref{rem:types_tria} to $\calD$, getting
$$
V(\partial\calD)=\mv(\calT)+\nI-\nII\geqslant\left\{
  \begin{array}{ll}
    \mv(\calT)\geqslant 5 & \mbox{if $\nI=0$} \\[2pt]
    \mv(\calT)+1-(V(\inter(\calD))-1)\geqslant 5 & \mbox{if $\nI>0$}
  \end{array}
\right..
$$
Hence, we have $V(\calD)\geqslant 6$, and then $V(\calT)\geqslant
V(\calD)+V(\calD_1)=12$: a contradiction to the hypothesis $V(\calT)\leqslant
11$.
\end{proof}

The two propositions above obviously yield the following result.
\begin{cor}\label{cor:1-2_tria}
  If $\calT$ is a root and $V(\calT)\leqslant 11$, then each minimal
  decomposition of $\calT$ is of type $(\calG,\calD,\emptyset)$ or of type
  $(\calG,\calD,\{\calD_1\})$ where $\calD_1$ is made up of one or two
  triangles.
\end{cor}

\paragraph{Computational results}
The computer program {\tt trialistgs11} implementing the algorithm described in
Section~\ref{sec:listing_algorithm}, specialized for the 11-vertex case with the results of
Section~\ref{sec:restrictions} and of this section, can be found
in~\cite{amendola:sito}.
Such results have simplified the search: for instance,
\begin{itemize}
\item by Proposition~\ref{prop:length_D1}, we search only for genus-surfaces
  with either one or two boundary components, and in the latter case
  one of the components must contain at most 4 vertices;
\item by Corollary~\ref{cor:1-2_tria}, there are only two cases for the
  triangulated disc $\calD_1$ (when $\partial\calG$ has two components);
\item by Proposition~\ref{prop:lenght_bnd_condition}, we can discard the
  genus-surfaces not fulfilling some properties.
\end{itemize}

The numbers of genus-surfaces found with at most 10 vertices are listed in
Table~\ref{tab:gs_numbers}, while the numbers of roots and non-roots with at most 11
vertices are listed in Table~\ref{tab:tria_numbers}.
\begin{table}
\begin{center}
\begin{small}
\begin{tabular*}{8cm}{@{}l@{\extracolsep{12pt}}l@{\extracolsep{12pt}}r@{\extracolsep{\fill}}l@{\extracolsep{12pt}}l@{\extracolsep{12pt}}r@{}}
  \toprule\\[-3mm]
$V$ & $S$        &  Number & $V$ & $S$        &  Number \\
\cmidrule{1-3}\cmidrule{4-6}
  3 & $S^2$      &       1 &   9 & $T^2$      &     230 \\
    &            &         &     & $S^+_2$    &    1261 \\
  5 & $\matRP^2$ &       1 &     & $S^+_3$    &      59 \\
    &            &         &     & $\matRP^2$ &      28 \\
  6 & $T^2$      &       1 &     & $K^2$      &     597 \\
    & $\matRP^2$ &       2 &     & $S^-_3$    &    6919 \\
    &            &         &     & $S^-_4$    &   18166 \\
  7 & $T^2$      &       5 &     & $S^-_5$    &   18199 \\
    & $\matRP^2$ &       6 &     & $S^-_6$    &    4994 \\
    & $K^2$      &      10 &     & $S^-_7$    &      78 \\
    &            &         &     &            &         \\
  8 & $T^2$      &      46 &  10 & $T^2$      &    1513 \\
    & $\matRP^2$ &      11 &     & $S^+_2$    &   50878 \\
    & $K^2$      &     108 &     & $S^+_3$    &   99177 \\
    & $S^-_3$    &     284 &     & $S^+_4$    &    3892 \\
    & $S^-_4$    &     134 &     & $\matRP^2$ &     356 \\
    & $S^-_5$    &       3 &     & $K^2$      &    3864 \\
    &            &         &     & $S^-_3$    &   82588 \\
    &            &         &     & $S^-_4$    &  713714 \\
    &            &         &     & $S^-_5$    & 3006044 \\
    &            &         &     & $S^-_6$    & 5672821 \\
    &            &         &     & $S^-_7$    & 4999850 \\
    &            &         &     & $S^-_8$    & 1453490 \\
    &            &         &     & $S^-_9$    &   53484 \\[1mm]
\bottomrule
  \end{tabular*}
\end{small}
\end{center}
\caption{Number of genus-surfaces with at most 10 vertices, used to list
  closed triangulated surfaces with at most 11 vertices,
  depending on the number of vertices $V$ and on the closed surface $S$ obtained by
  gluing discs to their boundary.}
  \label{tab:gs_numbers}
\end{table}
It is worth noting that we are searching for closed triangulated surfaces with at
most 11 vertices, hence we get a list of the genus-surfaces needed to construct
those triangulated surfaces.
If we had searched for closed triangulated surfaces with less vertices, we
would have found a shorter list of genus-surfaces.

The computer program carries out the search for a fixed homeomorphism type
({\rm i.e.}~genus and orientability) of the surface each time.
The longest case is the $S^-_{7}$-case taking $12$~days on a $2.33$GHz
notebook-processor to obtain the list.

\vspace{24pt} \noindent\textbf{Acknowledgments.}\small \quad
I am very grateful to Prof.~J\"urgen Bokowski and Simon King for the useful
discussions I have had in the beautiful period I have spent at the Department
of Mathematics in Darmstadt.
I would like to thank the Galileo Galilei Doctoral School of Pisa and the DAAD
(Deutscher Akademischer Austausch Dienst) for giving me the opportunity to
stay in Darmstadt, and Prof.~Alexander Martin for his willingness.
I would also like to thank the referee for the useful comments and corrections.

This paper is dedicated to Paolo.

\begin{small}

\end{small}


\begin{thebibliography}{99}

\bibitem{amendola:sito}
\textsc{G.~Amendola},
{\em \texttt{trialistgs11}},
\url{http://www.dm.unipi.it/~amendola/files/software/trialistgs11/}.

\bibitem{BarnetteEdelson1988}
\textsc{D.~W. Barnette -- A.~L. Edelson},
\textit{All $2$-manifolds have finitely many minimal triangulations},
Isr. J. Math. \textbf{67} (1988), 123--128.

\bibitem{Bokowski-Guedes_de_Oliveira2000}
\textsc{J.~Bokowski -- A.~Guedes de Oliveira},
\textit{On the generation of oriented matroids},
Discrete Comput.\ Geom. \textbf{24} (2000), 197--208.

\bibitem{Brueckner1897}
\textsc{M.~Br\"uckner},
\textit{Geschichtliche Bemerkungen zur Aufz\"ahlung der Vielflache},
Pr.\ Realgymn.\ Zwickau. \textbf{578} (1897).

\bibitem{Datta1999}
\textsc{B.~Datta},
\textit{Two dimensional weak pseudomanifolds on seven vertices},
Bol.\ Soc.\ Mat.\ Mex. III.\ Ser. \textbf{5} (1999), 419--426.

\bibitem{DattaNilakantan2002}
\textsc{B.~Datta -- N.~Nilakantan},
\textit{Two-dimensional weak pseudomanifolds on eight vertices},
Proc.\ Indian Acad.\ Sci.\ (Math.\ Sci.) \textbf{112} (2002), 257--281.

\bibitem{Duke1970}
\textsc{R.~A.~Duke},
\textit{Geometric embedding of complexes},
Am.\ Math.\ Mon. \textbf{77} (1970), 597--603.

\bibitem{Heawood1890}
\textsc{P.~J.~Heawood},
\textit{Map-colour theorem},
Quart.\ J.\ Pure Appl.\ Math. \textbf{24} (1890), 332--338.

\bibitem{Hougardy-Lutz-Zelke2006pre}
\textsc{S.~Hougardy -- F.~H. Lutz -- M.~Zelke},
\textit{Surface realization with the intersection edge functional},
\url{arXiv:math.MG/0608538} (2006), 19 pages.

\bibitem{Jungerman-Ringel1980}
\textsc{M.~Jungerman -- G.~Ringel},
\textit{Minimal triangulations on orientable surfaces}, Acta
Math. \textbf{145} (1980), 121--154.

\bibitem{Lutz:10vert}
\textsc{F.~H.~Lutz},
\textit{Enumeration and random realization of triangulated surfaces},
\url{arXiv:math.CO/0506316v2} (2005), 18 pages;
to appear in Discrete Differential Geometry (A. I. Bobenko, J. M. Sullivan,
P. Schr\"oder, and G. M. Ziegler, eds.), Oberwolfach Seminars, Birkh\"auser,
Basel.

\bibitem{Lutz:sito}
\textsc{F.~H.~Lutz},
\textit{The Manifold Page},
\url{http://www.math.tu-berlin.de/diskregeom/stellar/}.

\bibitem{Lutz-Sulanke:12vert}
\textsc{F.~H.~Lutz -- T.~Sulanke},
\textit{Isomorphism free lexicographic enumeration of triangulated surfaces
  and $3$-manifolds},
\url{arXiv:math.CO/0610022} (2006), 20 pages;
to appear in Eur. J. Comb.

\bibitem{Rado1925}
\textsc{T.~Rad\'o},
\textit{\"Uber den Begriff der Riemannschen Fl\"ache},
Acta Univ. Szeged \textbf{2} (1925), 101--121.

\bibitem{Ringel1955}
\textsc{G.~Ringel},
\textit{Wie man die geschlossenen nichtorientierbaren Fl\"achen in
  m\"oglichst wenig Dreiecke zerlegen kann},
Math. Ann. \textbf{130} (1955), 317--326.

\bibitem{Schewe2006pre}
\textsc{L.~Schewe},
2006, work in progress.

\bibitem{Steinitz1922}
\textsc{E.~Steinitz},
\textit{Polyeder und Raumeinteilungen}, Encyklop\"adie der mathematischen
Wissenschaften mit Einschluss ihrer Anwendungen, Dritter Band: Geometrie, {\rm
  III.1.2., Heft~9} (W.~Fr. Meyer and H.~Mohrmann, eds.), B.~G.~Teubner,
Leipzig, 1922, pp.~1--139.

\bibitem{Steinitz-Rademacher1934}
\textsc{E.~Steinitz -- H.~Rademacher},
``Vorlesungen \"uber die Theorie der Polyeder unter Einschlu\ss\ der Elemente
der Topologie'',
Reprint der 1934 Auflage, Grundlehren der mathematischen Wissenschaften,
vol.~41, Springer-Verlag, Berlin-New York, 1976, viii+351 pages.

\bibitem{Sulanke2006:pre}
\textsc{T.~Sulanke},
\textit{Irreducible triangulations of low genus surfaces},
\url{arXiv:math.CO/0606690} (2006), 10 pages.

\bibitem{Sulanke:sito}
\textsc{T.~Sulanke},
\textit{Numbers of triangulated surfaces},
\url{http://hep.physics.indiana.edu/~tsulanke/graphs/surftri/counts.txt}.

\end{thebibliography}
\end{document}